\documentclass[12pt]{amsart}
\usepackage[english]{babel}
\usepackage{amsmath}
\usepackage{amsthm}
\usepackage{anysize}
\usepackage{mathtools}
\usepackage{mathrsfs}
\usepackage{amssymb}
\usepackage{xcolor}
\usepackage{xfrac}
\usepackage{esint}
\usepackage{graphicx}
\usepackage{float}
\usepackage{array}
\usepackage{enumitem}
\usepackage[new]{old-arrows}
\usepackage[all,cmtip]{xy}
\usepackage{tikz-cd}
\usepackage[hidelinks]{hyperref}

\newtheorem{theorem}{Theorem}[section]

\newtheorem{lemma}[theorem]{Lemma}
\newtheorem{proposition}[theorem]{Proposition}

\DeclareMathOperator{\vol}{\mathbf{vol}}

\begin{document}
	\title{Existence of Infinite Product Measures} 
	\author{Juan Carlos Sampedro} \thanks{The author has been supported by the Research Grant PGC2018-097104-B-I00 of the Spanish Ministry of Science, Technology and Universities and by the PhD Grant PRE2019\_1\_0220 of the Basque Country Government.}
	\address{Institute of Interdisciplinary Mathematics \\
		Department of Mathematical Analysis and Applied Mathematics \\
		Complutense University of Madrid \\
		28040-Madrid \\
		Spain.}
	\email{juancsam@ucm.es}

\begin{abstract}
		A construction of product measures is given for an arbitrary sequence of measure spaces via outer measure techniques without imposing any condition on the underlying measure spaces. This result generalizes the ones given up to date.
\end{abstract}


\maketitle

\section{Introduction}

In this article, we go back to a central issue in the foundational period of the measure theory in the beginning of the 20th century, whose importance has been declining due to the production of certain partial results sufficienty general for its application. We talk about the construction of product measures on infinite product spaces. We start by presenting in a compact way the results obtained up to date.
\par 
The classical theory of product measures deals with two measure spaces $(X,\Sigma_{X},\mu_{X})$ and $(Y,\Sigma_{Y},\mu_{Y})$ in order to construct the product measure space $(X\times Y,\Sigma_{X}\otimes \Sigma_{Y},\mu_{X}\otimes\mu_{Y})$, where $\Sigma_{X}\otimes \Sigma_{Y}$ is the $\sigma$-algebra generated by $\mathcal{R}_{X\times Y}:=\left\{A\times B : A\in \Sigma_{X}, \ B\in\Sigma_{Y}\right\}$ and $\mu_{X}\otimes\mu_{Y}$ is a measure on $\Sigma_{X}\otimes \Sigma_{Y}$ satisfying the identity
\begin{equation}\label{E1}
(\mu_{X}\otimes\mu_{Y})(A\times B)=\mu_{X}(A)\cdot\mu_{Y}(B) \text{ for every } A\in\Sigma_{X}, B\in\Sigma_{Y}.
\end{equation}
\noindent The most common method to prove the existence of this measure is through the celebrated Caratheodory extension theorem as follows. Denote by $\mathcal{U}(\mathcal{R}_{X\times Y})$ the family of finite unions of elements of $\mathcal{R}_{X\times Y}$. It is easy to verify that $\mathcal{U}(\mathcal{R}_{X\times Y})$ is an algebra of subsets of $X\times Y$ and that every element of $\mathcal{U}(\mathcal{R}_{X\times Y})$ can be written as a finite union of pairwise disjoint members of $\mathcal{R}_{X\times Y}$. Define the set function
\begin{equation*}
\mu_{0}: \mathcal{U}(\mathcal{R}_{X\times Y}) \longrightarrow [0,\infty], \quad
 \biguplus_{i=1}^{N}A_{i}\times B_{i} \mapsto \sum_{i=1}^{N}\mu_{X}(A_{i})\cdot\mu_{Y}(B_{i}).
\end{equation*}
\noindent Along this paper, the notation $\uplus$ will mean that the union is a union of pairwise disjoint sets. It is classical (see e.g. \cite{B,F}), that the function $\mu_{0}$ is well defined and $\sigma$-additive on $\mathcal{U}(\mathcal{R}_{X\times Y})$. Hence by Caratheodory extension theorem, there exists a measure $\mu$ on the $\sigma$-algebra $\Sigma_{X}\otimes\Sigma_{Y}$ that extends $\mu_{0}$. Therefore, $\mu$ satisfies identity \eqref{E1} for each $A\in\Sigma_{X}$ and $B\in\Sigma_{Y}$. Moreover if $\mu_{X}$ and $\mu_{Y}$ are $\sigma$-finite, $\mu$ is the unique measure on $(X\times Y,\Sigma_{X}\otimes\Sigma_{Y})$ satisfying \eqref{E1}.
\par
Consider now a sequence $\{(\Omega_{i},\Sigma_{i},\mu_{i})\}_{i\in\mathbb{N}}$ of measure spaces. The problem now is to generalize the classical product measure construction to countable many, i.e., to construct the infinite product measure space
$\left(\bigtimes_{i\in\mathbb{N}}\Omega_{i},\bigotimes_{i\in\mathbb{N}}\Sigma_{i},\bigotimes_{i\in\mathbb{N}}\mu_{i}\right)$ where $\bigotimes_{i\in\mathbb{N}}\Sigma_{i}$ is the $\sigma$-algebra generated by the cylinder sets
\begin{equation*}
\mathcal{C}(\Sigma_{i})_{i\in\mathbb{N}}:=\left\{\bigtimes_{i=1}^{m}C_{i}\times\bigtimes_{i=m+1}^{\infty}\Omega_{i} :C_{i}\in\Sigma_{i}, \ \forall i\in\{1,2,...,m\} \text{ and } m\in\mathbb{N}\right\},
\end{equation*}
\noindent and $\bigotimes_{i\in\mathbb{N}}\mu_{i}$ is a measure on $\left(\bigtimes_{i\in\mathbb{N}}\Omega_{i},\bigotimes_{i\in\mathbb{N}}\Sigma_{i}\right)$ satisfying an analogue of identity \eqref{E1} for this general setting. For instance, if the finiteness condition $\prod_{i\in\mathbb{N}}\mu_{i}(\Omega_{i})\in [0,\infty)$ holds, the measure $\bigotimes_{i\in \mathbb{N}}\mu_{i}$ should satisfy the identity
\begin{equation}\label{E2}
\bigotimes_{i\in\mathbb{N}}\mu_{i}\left(\bigtimes_{i=1}^{m}C_{i}\times\bigtimes_{i=m+1}^{\infty}\Omega_{i}\right)=\prod_{i=1}^{m}\mu_{i}(C_{i})\cdot\prod_{i=m+1}^{\infty}\mu_{i}(\Omega_{i}),
\end{equation}
\noindent for each $\mathscr{C}=\bigtimes_{i=1}^{m}C_{i}\times\bigtimes_{i=m+1}^{\infty}\Omega_{i}\in\mathcal{C}(\Sigma_{i})_{i\in\mathbb{N}}$. Note that along this paper we use the following definition of infinite product for a sequence $(a_{i})_{i\in\mathbb{N}}\subset [0,\infty)$,
\begin{equation*}
\prod_{i\in\mathbb{N}}a_{i}:=\lim_{n\to\infty}\prod_{i=1}^{n}a_{i}.
\end{equation*}
The first attempt to address this problem was for the particular case of probability spaces. In 1933, A. Kolmogoroff proved in \cite{K2} the existence of a probability measure $\bigotimes_{i\in\mathbb{N}} m_{[0,1]}$ on the measurable space $\left([0,1]^{\mathbb{N}},\bigotimes_{i\in\mathbb{N}}\mathcal{B}_{[0,1]}\right)$, where $\mathcal{B}_{[0,1]}$ and $m_{[0,1]}$ are the Borel $\sigma$-algebra of $[0,1]$ and the Lebesgue measure of $[0,1]$ respectively, such that identity \eqref{E2} holds for every $\mathscr{C}\in \mathcal{C}(\mathcal{B}_{[0,1]})_{i\in\mathbb{N}}$.
		\noindent Kolmogoroff's proof was based on the compactness of the product space $[0,1]^{\mathbb{N}}$, provided by the Tychonoff result from 1930, \cite{T}. More general cases were discussed by Z. Lomnicki and S. Ulam in 1934 in the reference \cite{LU}.  In 1943, S. Kakutani generalized for general probability spaces the results of Kolmogoroff, Lomnicki and Ulam proving in \cite{K1} the following celebrated result.
		
		\begin{theorem}\label{KT}
			Given $\{(\Omega_{i},\Sigma_{i},\mu_{i})\}_{i\in\mathbb{N}}$ a family of probability spaces, there exists a unique probability measure $\bigotimes_{i\in\mathbb{N}}\mu_{i}$ on the measurable space $(\bigtimes_{i\in\mathbb{N}}\Omega_{i},\bigotimes_{i\in\mathbb{N}}\Sigma_{i})$ satisfying \eqref{E2} for every $\mathscr{C}\in \mathcal{C}(\Sigma_{i})_{i\in\mathbb{N}}$.
		\end{theorem}
		\noindent Kakutani's proof of this result has become standard in Probability and Measure Theory. The key tool of the proof relies on a result of E. Hopf (cf. \cite{H}, \cite[Theorem 3.2]{Y}). In 1996, S. Saeki gives in \cite{S1} a new proof of Theorem \ref{KT}, proving it in a more natural terms without the use of Hopf's result. It must be observed that if the measure spaces $\{(\Omega_{i},\Sigma_{i},\mu_{i})\}_{i\in\mathbb{N}}$ are not of probability but satisfy the finiteness condition $\prod_{i\in\mathbb{N}}\mu_{i}(\Omega_{j})\in[0,\infty)$, then normalizing each measure space, it can be also proven as a rather direct consequence of Theorem \ref{KT}, the existence of a unique measure on $(\bigtimes_{i\in\mathbb{N}}\Omega_{i},\bigotimes_{i\in \mathbb{N}}\Sigma_{i})$ satisfying identity \eqref{E2}.
		\par
	On other hand, if $\prod_{i\in\mathbb{N}}\mu_{i}(\Omega_{i})=\infty$, identity \eqref{E2} is no longer useful for the construction of infinite product measures since its value is always infinite. Nevertheless, we can ask the measure to verify the new identity
	\begin{equation}\label{E3}
	\bigotimes_{i\in \mathbb{N}}\mu_{i}\left(\bigtimes_{i\in\mathbb{N}}C_{i}\right)=\prod_{i\in\mathbb{N}}\mu_{i}(C_{i})
	\end{equation}
	\noindent for each $\bigtimes_{i\in\mathbb{N}}C_{i}\in\mathcal{F}(\Sigma_{i},\mu_{i})_{i\in\mathbb{N}}$, where $\mathcal{F}(\Sigma_{i},\mu_{i})_{i\in\mathbb{N}}$ is the set of \textit{finite rectangles} on $\bigtimes_{i\in\mathbb{N}}\Omega_{i}$ defined by
	\begin{equation*}
	\mathcal{F}(\Sigma_{i},\mu_{i})_{i\in\mathbb{N}}:=\left\{\bigtimes_{i\in\mathbb{N}}C_{i}\in \mathcal{R}(\Sigma_{i},\mu_{i})_{i\in\mathbb{N}} : \prod_{i\in\mathbb{N}}\mu_{i}(C_{i})\in[0,\infty) \right\},
	\end{equation*}
	\noindent where $\mathcal{R}(\Sigma_{i},\mu_{i})_{i\in\mathbb{N}}$ is the set of \textit{rectangles} defined by
	\begin{equation*}
	\mathcal{R}(\Sigma_{i},\mu_{i})_{i\in\mathbb{N}}:=\left\{\bigtimes_{i\in\mathbb{N}}C_{i} : C_{i}\in\Sigma_{i},\ \forall i\in\mathbb{N}\right\}.
	\end{equation*}
	Therefore, the natural extension of the classical theory to the non-finite case is the product measure space $\left(\bigtimes_{i\in\mathbb{N}}\Omega_{i},\bigotimes_{i\in\mathbb{N}}\Sigma_{i},\bigotimes_{i\in\mathbb{N}}\mu_{i}\right)$ where $\bigotimes_{i\in\mathbb{N}}\mu_{i}$ is a measure satisfying \eqref{E3}. 
	The construction of this measures had not been ignored in the last century and several attempts had been made trying to formalize it.
	\par
	 In 1963, E. O. Elliott and A. P. Morse published a paper \cite{EM} constructing this kind of product spaces through a reformulation of the classical infinite product called \textit{plus product}. Let $\mathfrak{a}=(a_{n})_{n\in\mathbb{N}}\subset [0,\infty]$ be a sequence of extended non-negative real numbers and $\mathscr{A}(\mathfrak{a})=\{n\in\mathbb{N}:a_{n}>1\}$. Elliott and Morse defined the \textit{plus product} of the sequence $(a_{n})_{n\in\mathbb{N}}$ by
\begin{equation*}
{\prod_{n\in\mathbb{N}}}^{+}a_{n}:=\prod_{n\in\mathscr{A}(\mathfrak{a})}a_{n}\cdot\prod_{n\notin\mathscr{A}(\mathfrak{a})}a_{n},
\end{equation*}
\noindent with the convention $0\cdot \infty=\infty \cdot 0=0$ and setting the value of the empty product to 1. The main purpose of defining this concept lies in the fact that the plus product exists for every sequence $(a_{n})_{n\in\mathbb{N}}$ of extended non-negative real numbers, which does not happen with the classic product, take for instance the sequence defined by

\begin{equation*}
a_{n}:=\left\{
\begin{array}{ll}
 2 & \text{if } n\equiv 0\mod 2 \\
\frac{1}{2} & \text{if } n\equiv 1\mod 2.
\end{array}
\right.
\end{equation*}

\noindent for each $n\in\mathbb{N}$. The classical and the plus product do not, in general, coincide, but if $(a_{n})_{n\in\mathbb{N}}\subset (0,\infty)$ and $\prod_{n\in\mathbb{N}}a_{n}$ exists, then $\prod_{n\in\mathbb{N}}^{+}a_{n}=\prod_{n\in\mathbb{N}}a_{n}$. We define the set of \textit{finite plus rectangles} by

		\begin{equation*}
		\mathcal{F}^{+}(\Sigma_{i},\mu_{i})_{i\in\mathbb{N}}:=\left\{\bigtimes_{i\in\mathbb{N}}C_{i}\in\mathcal{R}(\Sigma_{i},\mu_{i})_{i\in\mathbb{N}} : {\prod_{i\in\mathbb{N}}}^{+}\mu_{i}(C_{i})\in[0,\infty) \right\}.
		\end{equation*}
		Elliott and Morse proved that given a family of measure spaces $\{(\Omega_{i},\Sigma_{i},\mu_{i})\}_{i\in\mathbb{N}}$, there exists a measure $\lambda_{EM}$ on the measurable space $\left(\bigtimes_{i\in\mathbb{N}}\Omega_{i}, \bigotimes_{i\in\mathbb{N}} \Sigma_{i}\right)$
			satisfying for every $\mathscr{C}=\bigtimes_{i\in\mathbb{N}}C_{i}\in\mathcal{F}^{+}(\Sigma_{i},\mu_{i})_{i\in\mathbb{N}}$, the following identity 
			\begin{equation}\label{E4}
			\lambda_{EM}(\mathscr{C})={\prod_{i\in\mathbb{N}}}^{+}\mu_{i}(C_{i}).
			\end{equation}
			Using earlier observation, if a finite plus rectangle satisfies ${\prod}^{+}_{i}\mu_{i}(C_{i})\neq 0$ (which implies that $\mu_{i}(C_{i})>0$ for every $i\in\mathbb{N}$) and its classical product $\prod_{i}\mu_{i}(C_{i})$ exists, then its plus product must coincide with the classical product $\prod_{i}\mu_{i}(C_{i})$. However, if ${\prod}^{+}_{i}\mu_{i}(C_{i})= 0$, the products do not, in general, coincide.  In consequence, this result does not stablish the existence of the required product measure since there are substantial sequences satisfying ${\prod}^{+}_{i}\mu_{i}(C_{i})= 0$ with nonzero classical product, take for instance the sequence $a_{n}=\exp((-1)^{n+1}/n)$ for each $n\in\mathbb{N}$.
		\par
		In 2004, R. Baker, proved in \cite[Theorem I]{RB2} (see also \cite{RB}), the existence of the required product measure for the particular case in which the involved spaces $\Omega_{i}$ are locally compact metric spaces, $\Sigma_{i}$ is the Borel $\sigma$-algebra of $\Omega_{i}$ and the property $\mathcal{D}$ is satisfied.
		\begin{enumerate}
			\item[($\mathcal{D}$)] For every $i\in\mathbb{N}$ and $\delta>0$, there exists a sequence $\{A_{j}\}_{j\in\mathbb{N}}\subset\Sigma_{i}$ such that $\text{diam}(A_{j})<\delta$ for each $j\in\mathbb{N}$ and 
			$$\Omega_{i}=\bigcup_{j\in\mathbb{N}}A_{j}.$$
		\end{enumerate}
		He proved that given a sequence $\{(\Omega_{i},\Sigma_{i})\}_{i\in\mathbb{N}}$ satisfying the previous assumptions and a family of regular Borel measures $\{\mu_{i}\}_{i\in\mathbb{N}}$, each $\mu_{i}$ defined in the measurable space $(\Omega_{i},\Sigma_{i})$, there exists a measure $\lambda_{B}$ on the measurable space $(\bigtimes_{i\in\mathbb{N}}\Omega_{i}, \bigotimes_{i\in\mathbb{N}}\Sigma_{i})$ satisfying \eqref{E3} for every $\mathscr{C}\in\mathcal{F}(\Sigma_{i},\mu_{i})_{i\in\mathbb{N}}$.
        \par In 2005, P. A. Loeb and D. A. Ross gave in \cite[Theorem 1.1]{LR} another attempt to formalize the product measure via Nonstandard Analysis techniques and Loeb Measures \cite[\textsection 4]{A}. They established that given a sequence of Hausdorff topological spaces $\Omega_{i}$ and a corresponding sequence of regular Borel measure spaces $\{(\Omega_{i},\Sigma_{i},\mu_{i})\}_{i\in\mathbb{N}}$, there exists a measure $\lambda_{LR}$ on the measurable space $(\bigtimes_{i\in\mathbb{N}}\Omega_{i}, \bigotimes_{i\in\mathbb{N}}\Sigma_{i})$ such that identity \eqref{E3} holds for every $\bigtimes_{i\in\mathbb{N}}C_{i}\in \mathcal{F}(\Sigma_{i},\mu_{i})_{i\in\mathbb{N}}$, provided that each $C_{i}$ is compact in $\Omega_{i}$.
		\par Finally, in 2011 G. R. Pantsulaia presented in \cite[Theorem 3.10]{P1} probably the best generalization of product measures to countable many up to date. He proved the following.
		\begin{theorem}\label{PT}
			Let $\{(\Omega_{i},\Sigma_{i},\mu_{i})\}_{i\in\mathbb{N}}$ be a family of $\sigma$-finite measure spaces. Then there exists a measure $\lambda_{P}$ on $\left(\bigtimes_{i\in\mathbb{N}}\Omega_{i},\bigotimes_{i\in\mathbb{N}}\Sigma_{i}\right)$ satisfying identity \eqref{E3} for each $\mathscr{C}\in\mathcal{F}(\Sigma_{i},\mu_{i})_{i\in\mathbb{N}}$.
		\end{theorem} 
		This result provides a standard proof of the existence of the product measure without imposing any topological assumption over the involved spaces, in contrast with the constructions of Baker and Loeb \& Ross. However, the construction of Pantsulaia requires the measure spaces $(\Omega_{i},\Sigma_{i},\mu_{i})$ to be $\sigma$-finite for every $i\in\mathbb{N}$ and the constructed measure $\lambda_{P}$ is not the restriction of an outer measure. The last condition is not necessary but it is rather natural and facilitates several computations like Fubini theorem's proof.
		\par
		The aim of this article is to write an additional chapter to this story by constructing the product measure $\bigotimes_{i\in \mathbb{N}}\mu_{i}$ for any sequence of measure spaces $\{(\Omega_{i},\Sigma_{i},\mu_{i})\}_{i\in\mathbb{N}}$, without imposing any restriction on them. This establishes an analogue Kakutani's Theorem \ref{KT} for general product measure spaces, concluding with an affirmative answer the problem of the existence of product measures. The main result of this article is the following.
		\begin{theorem}\label{J}
			Let $\{(\Omega_{i},\Sigma_{i},\mu_{i})\}_{i\in\mathbb{N}}$ be a family of measure spaces. Then there exists a measure $\bigotimes_{i\in\mathbb{N}}\mu_{i}$ on the measurable space $(\bigtimes_{i\in\mathbb{N}}\Omega_{i},\bigotimes_{i\in\mathbb{N}}\Sigma_{i})$ satisfying identity \eqref{E3} for each $\mathscr{C}=\bigtimes_{i\in\mathbb{N}}C_{i}\in\mathcal{F}(\Sigma_{i},\mu_{i})_{i\in\mathbb{N}}$. That is, the identity
			\begin{equation*}
			\bigotimes_{i\in\mathbb{N}}\mu_{i}(\mathscr{C})=\prod_{i\in\mathbb{N}}\mu_{i}(C_{i}).
			\end{equation*}
		\end{theorem}
		Furthermore, as will be seen later, this measure is constructed in an elementary way using the Caratheodory extension theorem and greatly simplifies the proofs given by Baker, Loeb \& Ross and Pantsulaia.
		\par
		This paper is organized as follows. In section two, we give a result that states that the volume of finite rectangles $\mathcal{F}(\Sigma_{i},\mu_{i})_{i\in\mathbb{N}}$ behaves like the classical volume of cubes in finite dimensions. This result will be imperative for the construction of the required measure. In section three, we construct the product measure proving the main result of this article, Theorem \ref{J}.

\section{A Key Result}

In this section, we obtain the main tool that will be used to prove the existence of the required measure. This result states that the volume of finite rectangles behaves like the volume of finite dimensional rectangles. Let $\{(\Omega_{i},\Sigma_{i},\mu_{i})\}_{i\in\mathbb{N}}$ be an arbitrary family of measure spaces. We define the \textit{volume} map $\vol: \mathcal{F}(\Sigma_{i},\mu_{i})_{i\in\mathbb{N}}\to[0,\infty)$ by 
\begin{equation*}
\vol\left(\bigtimes_{i\in\mathbb{N}}C_{i}\right):=\prod_{i\in\mathbb{N}}\mu_{i}(C_{i}).
\end{equation*}
\noindent It will be proved that the $\vol$ map is a good choice for the extension of the classical notion of volume in finite dimensions. Recall that the notation $\uplus$ means that the union is a union of pairwise disjoint sets.

\begin{theorem}\label{T2}
	Let $\{(\Omega_{i},\Sigma_{i},\mu_{i})\}_{i\in\mathbb{N}}$ be a sequence of measure spaces, $\mathscr{C}\in \mathcal{F}(\Sigma_{i},\mu_{i})_{i\in\mathbb{N}}$ and $\{\mathscr{C}_{n}\}_{n\in\mathbb{N}}\subset \mathcal{F}(\Sigma_{i},\mu_{i})_{i\in\mathbb{N}}$. Then, the following statements hold:
	
	\begin{enumerate}
		\item If $\mathscr{C}=\biguplus_{n\in\mathbb{N}}\mathscr{C}_{n}$, then
		\begin{equation}\label{JC0}
		\vol(\mathscr{C})=\sum_{n\in\mathbb{N}}\vol(\mathscr{C}_{n}).
		\end{equation}
		\item If $\biguplus_{n\in\mathbb{N}}\mathscr{C}_{n}\subset\mathscr{C}$, then
		\begin{equation}\label{JC1}
		\sum_{n\in\mathbb{N}}\vol(\mathscr{C}_{n})\leq\vol(\mathscr{C}).
		\end{equation}
		\item If $\mathscr{C}\subset\bigcup_{n\in\mathbb{N}}\mathscr{C}_{n}$, then
		\begin{equation*}\label{JC2}
		\vol(\mathscr{C})\leq\sum_{n\in\mathbb{N}}\vol(\mathscr{C}_{n}).
		\end{equation*}
	\end{enumerate}
	
\end{theorem}

\noindent For the proof of Theorem \ref{T2}, it will be necessary to prove some lemmas.

\begin{lemma}\label{LL}
	Let $(a_{n})_{n\in\mathbb{N}}, (b_{n})_{n\in\mathbb{N}}$ be two sequences of non-negative real numbers such that
	\begin{enumerate}
		\item $a_{n}\leq b_{n}$ for each $n\in\mathbb{N}$,
		\item $\prod_{n\in\mathbb{N}}b_{n}\in[0,\infty)$.
	\end{enumerate}
	Then $\prod_{n\in\mathbb{N}}a_{n}$ is well defined,
	\begin{equation}\label{P1}
	\prod_{n\in\mathbb{N}}a_{n}\in[0,\infty),
	\end{equation}
	and
	\begin{equation}\label{INE}
	\prod_{n\in\mathbb{N}}a_{n}\leq\prod_{n\in\mathbb{N}}b_{n}.
	\end{equation}
\end{lemma}

\begin{proof}
	If $\prod_{n\in\mathbb{N}}b_{n}=0$, the result is evident. Suppose $\prod_{n\in\mathbb{N}}b_{n}\in(0,\infty)$. Then, since $(a_{n}/b_{n})_{n\in\mathbb{N}}\subset [0,1]$, the partial products are monotone and bounded by $1$. In consequence  $\prod_{n\in\mathbb{N}}\frac{a_{n}}{b_{n}}\in[0,\infty)$. By using the elementary properties of the limit we obtain
	\begin{equation*}
	\lim_{m\to\infty}\prod_{n=1}^{m}a_{n}=\lim_{m\to\infty}\prod_{n=1}^{m}b_{n}\prod_{n=1}^{m}\frac{a_{n}}{b_{n}}=\prod_{n\in\mathbb{N}}b_{n}\prod_{n\in\mathbb{N}}\frac{a_{n}}{b_{n}}\in[0,\infty).
	\end{equation*}
	This concludes the proof of \eqref{P1}. For \eqref{INE} just note that the partial products satisfies $\prod_{n=1}^{m} a_{n}\leq\prod_{n=1}^{m} b_{n}$ for each $m\in\mathbb{N}$.
\end{proof}

\begin{lemma}\label{L4}
		Let $\mathscr{C}_{1}\in\mathcal{F}(\Sigma_{i},\mu_{i})_{i\in\mathbb{N}}$ and $\mathscr{C}_{2}\in\mathcal{R}(\Sigma_{i},\mu_{i})_{i\in\mathbb{N}}$. Then $\mathscr{C}_{1}\cap\mathscr{C}_{2}\in\mathcal{F}(\Sigma_{i},\mu_{i})_{i\in\mathbb{N}}$. Moreover, if $\mathscr{C}_{2}\in\mathcal{F}(\Sigma_{i},\mu_{i})_{i\in\mathbb{N}}$, for each $i\in\{1,2\}$,
		\begin{equation}\label{INE2}
		\vol(\mathscr{C}_{1}\cap\mathscr{C}_{2})\leq\vol(\mathscr{C}_{i}).
		\end{equation}
\end{lemma}

\begin{proof}
	Firstly, let us denote
	\begin{equation*}
	\mathscr{C}_{1}=\bigtimes_{i\in\mathbb{N}}C_{1}^{i} \quad \text{ and } \quad \mathscr{C}_{2}=\bigtimes_{i\in\mathbb{N}}C_{2}^{i}.
	\end{equation*}
	\noindent It is apparent that
	\begin{equation*}
	\mathscr{C}_{1}\cap\mathscr{C}_{2}=\bigtimes_{i\in\mathbb{N}}C_{1}^{i}\cap C_{2}^{i}.
	\end{equation*}
	\noindent Since $\Sigma_{i}$ is a $\sigma$-algebra, $C_{1}^{i}\cap C_{2}^{i}\in\Sigma_{i}$ for each $i\in\mathbb{N}$. Moreover, since $\mu_{i}(C_{1}^{i}\cap C_{2}^{i})\leq \mu_{i}(C_{1}^{i})$ for each $i\in\mathbb{N}$ and $\prod_{i\in\mathbb{N}}\mu_{i}(C_{1}^{i})=\vol(\mathscr{C}_{1})\in[0,\infty)$, it follows from Lemma \ref{LL}, that
	\begin{equation*}
	\prod_{i\in\mathbb{N}}\mu_{i}(C_{1}^{i}\cap C_{2}^{i})\in[0,\infty).
	\end{equation*}
	\noindent Thus $\mathscr{C}_{1}\cap\mathscr{C}_{2}\in\mathcal{F}(\Sigma_{i},\mu_{i})_{i\in\mathbb{N}}$. Inequality \eqref{INE2} follows from \eqref{INE}.
\end{proof}

\begin{lemma}\label{LLL}
	Let $\mathscr{C}=\bigtimes_{i\in\mathbb{N}}C_{i}\in\mathcal{F}(\Sigma_{i},\mu_{i})_{i\in\mathbb{N}}$. For each $i\in\mathbb{N}$, let us denote $(C_{i},\Sigma_{C_{i}},\mu_{C_{i}})$ the restriction of the measure space $(\Omega_{i},\Sigma_{i},\mu_{i})$ to $C_{i}$. Then, the following statements hold 
	
	\begin{enumerate}
		\item $\mathcal{F}(\Sigma_{C_{i}},\mu_{C_{i}})_{i\in\mathbb{N}}\subset\mathcal{F}(\Sigma_{i},\mu_{i})_{i\in\mathbb{N}}$.
		\item If $\mathscr{D}\in\mathcal{F}(\Sigma_{C_{i}},\mu_{C_{i}})_{i\in\mathbb{N}}$ and $\mathscr{E}\in\mathcal{F}(\Sigma_{i},\mu_{i})_{i\in\mathbb{N}}$, then $$\mathscr{D}\cap\mathscr{E}\in \mathcal{F}(\Sigma_{C_{i}},\mu_{C_{i}})_{i\in\mathbb{N}}$$ 
		In particular, if $\mathscr{E}\in\mathcal{F}(\Sigma_{i},\mu_{i})_{i\in\mathbb{N}}$ and $\mathscr{E}\subset\mathscr{C}$, necessarily $\mathscr{E}\in\mathcal{F}(\Sigma_{C_{i}},\mu_{C_{i}})_{i\in\mathbb{N}}$.
		\item There exists a measure $\bigotimes_{i\in\mathbb{N}}\mu_{C_{i}}$ on the measurable space $(\mathscr{C},\bigotimes_{i\in\mathbb{N}}\Sigma_{C_{i}})$ satisfying for each $\mathscr{D}\in\mathcal{F}(\Sigma_{C_{i}},\mu_{C_{i}})_{i\in\mathbb{N}}$ the identity 
		\begin{equation}\label{CT}
		\bigotimes_{i\in \mathbb{N}}\mu_{C_{i}}\left(\mathscr{D}\right)=\vol\left(\mathscr{D}\right).
		\end{equation}
	\end{enumerate}
\end{lemma}

\begin{proof}
	Let $\mathscr{D}=\bigtimes_{i\in\mathbb{N}}D_{i}\in\mathcal{F}(\Sigma_{C_{i}},\mu_{C_{i}})_{i\in\mathbb{N}}$. Then $D_{i}\in\Sigma_{C_{i}}\subset\Sigma_{i}$ for each $i\in\mathbb{N}$ and
	\begin{equation*}
	\prod_{i\in\mathbb{N}}\mu_{i}(D_{i})=\prod_{i\in\mathbb{N}}\mu_{C_{i}}(D_{i})\in [0,\infty).
	\end{equation*}
	Therefore $\mathscr{D}\in\mathcal{F}(\Sigma_{i},\mu_{i})_{i\in\mathbb{N}}$. This proves the first statement. For the second one,  let us denote
	\begin{equation*}
	\mathscr{D}=\bigtimes_{i\in\mathbb{N}}D_{i}\in\mathcal{F}(\Sigma_{C_{i}},\mu_{C_{i}})_{i\in\mathbb{N}} \quad \text{and} \quad  \mathscr{E}=\bigtimes_{i\in\mathbb{N}}E_{i}\in\mathcal{F}(\Sigma_{i},\mu_{i})_{i\in\mathbb{N}}.
	\end{equation*}
	Then for each $i\in\mathbb{N}$, $D_{i}\cap E_{i}\in \Sigma_{C_{i}}$ since $D_{i}\in\Sigma_{C_{i}}$ and $D_{i}\cap \Sigma_{i}\subset C_{i}\cap \Sigma_{i}=\Sigma_{C_{i}}$, where we are using the notation
	\begin{equation*}
	A\cap \Sigma_{i}:=\{ A\cap B: B\in\Sigma_{i}\}
	\end{equation*}
	for $A\in\{D_{i},C_{i}\}$. On the other hand, since $\mathscr{D},\mathscr{E}\in\mathcal{F}(\Sigma_{i},\mu_{i})_{i\in\mathbb{N}}$, by Lemma \ref{L4}, $\mathscr{D}\cap\mathscr{E}\in\mathcal{F}(\Sigma_{i},\mu_{i})_{i\in\mathbb{N}}$ and
	\begin{equation*}
	\prod_{i\in\mathbb{N}}\mu_{C_{i}}(D_{i}\cap E_{i})=\prod_{i\in\mathbb{N}}\mu_{i}(D_{i}\cap E_{i})\in[0,\infty).
	\end{equation*}
	Therefore $\mathscr{D}\cap\mathscr{E}\in\mathcal{F}(\Sigma_{C_{i}},\mu_{C_{i}})_{i\in\mathbb{N}}$. This proves the second statement. Since the sequence of measure spaces $\{(C_{i},\Sigma_{C_{i}},\mu_{C_{i}})\}_{i\in\mathbb{N}}$ satisfies the finiteness condition 
	$$\prod_{i\in\mathbb{N}}\mu_{C_{i}}(C_{i})=\prod_{i\in\mathbb{N}}\mu_{i}(C_{i})\in[0,\infty),$$
	by a discussion after Theorem \ref{KT}, there exists a measure $\bigotimes_{i\in\mathbb{N}}\mu_{C_{i}}$ on the measurable space $(\mathscr{C},\bigotimes_{i\in\mathbb{N}}\Sigma_{C_{i}})$ satisfying for each $\mathscr{D}=\bigtimes_{i\in\mathbb{N}}D_{i}\in\mathcal{F}(\Sigma_{C_{i}},\mu_{C_{i}})_{i\in\mathbb{N}}$ the identity
	\begin{equation}\label{E7}
	\bigotimes_{i\in \mathbb{N}}\mu_{C_{i}}\left(\mathscr{D}\right)=\prod_{i\in\mathbb{N}}\mu_{C_{i}}(D_{i})=\prod_{i\in\mathbb{N}}\mu_{i}(D_{i})=\vol\left(\mathscr{D}\right).
	\end{equation}
	It should be noted that the set map $\vol$ is well defined over $\mathcal{F}(\Sigma_{C_{i}},\mu_{C_{i}})_{i\in\mathbb{N}}$ since $$\mathcal{F}(\Sigma_{C_{i}},\mu_{C_{i}})_{i\in\mathbb{N}}\subset\mathcal{F}(\Sigma_{i},\mu_{i})_{i\in\mathbb{N}},$$ which follows from item (1). This proves the third statement and finishes the proof.
\end{proof}

\begin{proof}[Proof of Theorem 2.1]
	Let us begin by proving items $(1)$ and $(2)$. In these cases, $\mathscr{C}_{n}\subset\mathscr{C}$ for each $n\in\mathbb{N}$ and hence by item 2 of Lemma \ref{LLL}, $\{\mathscr{C}_{n}\}_{n\in\mathbb{N}}\subset \mathcal{F}(\Sigma_{C_{i}},\mu_{C_{i}})$, where
	\begin{equation*}
	\mathscr{C}=\bigtimes_{i\in\mathbb{N}}C_{i}.
	\end{equation*} 
	By item (3) of Lemma \ref{LLL}, there exists a measure $\bigotimes_{i\in \mathbb{N}}\mu_{C_{i}}$ on $(\mathscr{C},\bigotimes_{i\in\mathbb{N}}\Sigma_{C_{i}})$ which coincides with the map \textbf{vol} over $\mathcal{F}(\Sigma_{C_{i}},\mu_{C_{i}})_{i\in\mathbb{N}}$. By the basic properties of the measures, we deduce identities \eqref{JC0} and \eqref{JC1}. Finally, suppose the hypothesis of case $(3)$ holds. It is apparent that
	\begin{equation}\label{IN}
	\mathscr{C}=\bigcup_{n\in\mathbb{N}}\mathscr{C}\cap\mathscr{C}_{n}\subset\bigcup_{n\in\mathbb{N}}\mathscr{C}_{n}.
	\end{equation}
	By Lemma \ref{L4}, $\mathscr{C}\cap\mathscr{C}_{n}\in\mathcal{F}(\Sigma_{i},\mu_{i})_{i\in\mathbb{N}}$ and $\vol(\mathscr{C}\cap\mathscr{C}_{n})\leq\vol(\mathscr{C}_{n})$ for each $n\in\mathbb{N}$. Therefore, by using again the measure of Lemma \ref{LLL}, item (3), for $\mathscr{C}$ and identity \eqref{IN}, we deduce 
	\begin{align*}
	\vol(\mathscr{C})&=\bigotimes_{i\in\mathbb{N}}\mu_{C_{i}}(\mathscr{C})=\bigotimes_{i\in\mathbb{N}}\mu_{C_{i}}\left(\bigcup_{n\in\mathbb{N}}\mathscr{C}\cap\mathscr{C}_{n}\right)\\
	& \leq\sum_{n\in\mathbb{N}}\bigotimes_{i\in\mathbb{N}}\mu_{C_{i}}(\mathscr{C}\cap\mathscr{C}_{n})
	=\sum_{n\in\mathbb{N}}\vol(\mathscr{C}\cap\mathscr{C}_{n})\leq\sum_{n\in\mathbb{N}}\vol(\mathscr{C}_{n}).
	\end{align*}
	This concludes the proof.
\end{proof}

\section{Construction of the Measure}

 For the proof of the existence of a product measure for an arbitrary family of measure spaces $\{(\Omega_{i},\Sigma_{i},\mu_{i})\}_{i\in\mathbb{N}}$, we consider the outer measure $\mu^{\ast}:\mathcal{P}(\bigtimes_{i\in\mathbb{N}}\Omega_{i})\to [0,\infty]$, defined by
\begin{equation}\label{OM}
\mu^{\ast}(A):=\text{inf}\left\{ \sum_{n\in\mathbb{N}}\vol(\mathscr{C}_{n}) : \{\mathscr{C}_{n}\}_{n\in\mathbb{N}}\subset\mathcal{F}(\Sigma_{i},\mu_{i})_{i\in\mathbb{N}} \text{ and } A\subset\bigcup_{n\in\mathbb{N}}\mathscr{C}_{n} \right\}
\end{equation}
\noindent for every $A\in\mathcal{P}(\bigtimes_{i\in\mathbb{N}}\Omega_{i})$, where we set $\inf \emptyset =\infty$. Along this article, the notation $\mathcal{P}(A)$ stands for the power set of a set $A$. It is straightforward to prove that $\mu^{\ast}$ defines an outer measure, see e.g. \cite[Chapter 1, Theorem 4]{R}. We will prove that this outer measure is, in fact, a measure on the $\sigma$-algebra $\bigotimes_{i\in \mathbb{N}}\Sigma_{i}$ and it satisfies identity \eqref{E3} for each $\mathscr{C}\in\mathcal{F}(\Sigma_{i},\mu_{i})_{i\in\mathbb{N}}$. We will make use of the following set theoretic lemma whose proof is straightforward.

\begin{lemma}\label{L5}
	Let $\mathscr{C}=\bigtimes_{i\in\mathbb{N}}C_{i}\in\mathcal{R}(\Sigma_{i},\mu_{i})$. Then
	\begin{equation}\label{E5}
	\mathscr{C}^{c}=\biguplus_{n\in\mathbb{N}}\left(\bigtimes_{i=1}^{n-1}C_{i}\times C_{n}^{c}\times\bigtimes_{i=n+1}^{\infty}\Omega_{i}\right).
	\end{equation}
\end{lemma}

\begin{theorem}[\textbf{Measurability}]\label{T2.7}
	Every set in $\mathcal{C}(\Sigma_{i})_{i\in\mathbb{N}}$ is $\mu^{\ast}$-measurable.
\end{theorem}

\begin{proof}
	Take $\mathscr{C}\in\mathcal{C}(\Sigma_{i})_{i\in\mathbb{N}}$ and $B\in\mathcal{P}(\bigtimes_{i\in\mathbb{N}}\Omega_{i})$. We have to prove that the inequality 	\begin{equation}\label{S}
	\mu^{\ast}(B)\geq \mu^{\ast}(B\cap\mathscr{C})+\mu^{\ast}(B\cap\mathscr{C}^{c})
	\end{equation}
	holds. If $\mu^{\ast}(B)=\infty$, then condition \eqref{S} holds. Hence assume that $\mu^{\ast}(B)<\infty$. Take any $\varepsilon>0$ and a family $\{\mathscr{B}_{n}\}_{n\in\mathbb{N}}\subset\mathcal{F}(\Sigma_{i},\mu_{i})_{i\in\mathbb{N}}$ such that
	\begin{equation*}
	B\subset\bigcup_{n\in\mathbb{N}}\mathscr{B}_{n}
	\end{equation*}
	\noindent and
	\begin{equation}\label{EQ9}
	\sum_{n\in\mathbb{N}}\vol(\mathscr{B}_{n})\leq \mu^{\ast}(B)+\varepsilon.
	\end{equation}
	\noindent Denoting $\mathscr{C}=\bigtimes_{j=1}^{m}C_{j}\times\bigtimes_{j=m+1}^{\infty}\Omega_{j}$, by Lemma \ref{L5}, we deduce that
	\begin{equation}
	\mathscr{C}^{c}=\biguplus_{i=1}^{m}\mathscr{C}_{i}, \quad \mathscr{C}_{i}:=\left(\bigtimes_{j=1}^{i-1}C_{j}\times C_{i}^{c}\times\bigtimes_{j=i+1}^{\infty}\Omega_{j}\right).
	\end{equation}
	As a direct consequence, we can decompose
	\begin{align*}
	\mathscr{B}_{n}&=(\mathscr{B}_{n}\cap\mathscr{C})\uplus(\mathscr{B}_{n}\cap\mathscr{C}^{c}) =(\mathscr{B}_{n}\cap\mathscr{C})\uplus\left(\biguplus_{i=1}^{m}(\mathscr{B}_{n}\cap \mathscr{C}_{i})\right),
	\end{align*}
	where $\{\mathscr{C}_{i}\}_{i\in\mathbb{N}}\subset\mathcal{C}(\Sigma_{i})_{i\in\mathbb{N}}$.
	By Lemma \ref{L4}, it follows that $\mathscr{B}_{n}\cap\mathscr{C},\mathscr{B}_{n}\cap\mathscr{C}_{i}\in\mathcal{F}(\Sigma_{i},\mu_{i})_{i\in\mathbb{N}}$ and therefore, by Theorem \ref{T2} item (1), we arrive to the identity
	\begin{equation*}
	\vol(\mathscr{B}_{n})=\vol(\mathscr{B}_{n}\cap\mathscr{C})+\sum_{i=1}^{m}\vol(\mathscr{B}_{n}\cap \mathscr{C}_{i}).
	\end{equation*}
	Hence, by equation \eqref{EQ9} and the definition of the outer measure $\mu^{\ast}$, \eqref{OM}, it becomes apparent that
	\begin{align*}
	\mu^{\ast}(B)+\varepsilon\geq\sum_{n\in\mathbb{N}}\vol(\mathscr{B}_{n})& =\sum_{n\in\mathbb{N}}\vol(\mathscr{B}_{n}\cap\mathscr{C})+\sum_{n\in\mathbb{N}}\sum_{i=1}^{m}\vol(\mathscr{B}_{n}\cap \mathscr{C}_{i}) \\
	&\geq \mu^{\ast}(B\cap\mathscr{C})+\mu^{\ast}(B\cap\mathscr{C}^{c}).
	\end{align*}
	\noindent The last step follows from the inclusion
	\begin{align*}
	B\cap\mathscr{C}^{c}\subset\left(\bigcup_{n\in\mathbb{N}}\mathscr{B}_{n}\right)\cap\mathscr{C}^{c}=\bigcup_{n\in\mathbb{N}}(\mathscr{B}_{n}\cap\mathscr{C}^{c})=\bigcup_{n\in\mathbb{N}}\biguplus_{i=1}^{m}(\mathscr{B}_{n}\cap \mathscr{C}_{i}).
	\end{align*}
	\noindent Taking $\varepsilon\to 0$, inequality \eqref{S} holds for every $B\in\mathcal{P}(\bigtimes_{i\in\mathbb{N}}\Omega_{i})$ and therefore $\mathscr{C}$ is $\mu^{\ast}$-measurable.
\end{proof}

\noindent In Theorem \ref{T2.7} we have proved that $\mathcal{C}(\Sigma_{i})_{i\in\mathbb{N}}$ is a subfamily of the Caratheodory $\sigma$-algebra $\mathfrak{C}$ associated to the outer measure $\mu^{\ast}$, and therefore by Caratheodory extension theorem, $\mu^{\ast}$ defines a measure on $\sigma(\mathcal{C}(\Sigma_{i})_{i\in\mathbb{N}})=\bigotimes_{i\in\mathbb{N}} \Sigma_{i}$. We will denote 
\begin{equation*}
\bigotimes_{i\in\mathbb{N}}\mu_{i}:=\mu^{\ast}\big |_{\bigotimes_{i\in\mathbb{N}} \Sigma_{i}}.
\end{equation*}
Finally, we will prove that the outer measure $\mu^{\ast}$ satisfies identity \eqref{E3} for each $\mathscr{C}\in\mathcal{F}(\Sigma_{i},\mu_{i})_{i\in\mathbb{N}}$.

\begin{proposition}[\textbf{Volume}]\label{L14}
	For each $\mathscr{C}\in\mathcal{F}(\Sigma_{i},\mu_{i})_{i\in\mathbb{N}}$, the following identity holds
	\begin{equation*}
	\mu^{\ast}(\mathscr{C})=\vol(\mathscr{C}).
	\end{equation*}
\end{proposition}

\begin{proof}
	 Let $\{\mathscr{C}_{n}\}_{n\in\mathbb{N}}\subset\mathcal{F}(\Sigma_{i},\mu_{i})_{i\in\mathbb{N}}$ be a cover of $\mathscr{C}$, i.e.,
	\begin{equation*}
	\mathscr{C}\subset\bigcup_{n\in\mathbb{N}}\mathscr{C}_{n}.
	\end{equation*}
	\noindent By Theorem \ref{T2} item (3), we deduce the following inequality
	\begin{equation*}
	\vol(\mathscr{C})\leq\sum_{n\in\mathbb{N}}\vol(\mathscr{C}_{n}).
	\end{equation*}
	\noindent Then, taking the infimum over all such covers, we stablish that $\vol(\mathscr{C})\leq\mu^{\ast}(\mathscr{C})$. Finally, considering the particular cover $\{\mathscr{C}_{n}\}_{n\in\mathbb{N}}\subset\mathcal{F}(\Sigma_{i},\mu_{i})_{i\in\mathbb{N}}$ defined by
	\begin{equation*}
	\mathscr{C}_{n}:=\left\{
	\begin{array}{ll}
	\mathscr{C} & \text{ if } n= 1\\
	\emptyset &\text{ if } n\neq 1
	\end{array}\right.
	\end{equation*}
	\noindent it follows from the definition of $\mu^{\ast}$ that
	\begin{equation*}
	\mu^{\ast}(\mathscr{C})\leq\vol(\mathscr{C})\leq\mu^{\ast}(\mathscr{C})
	\end{equation*}
	\noindent which implies $\vol(\mathscr{C})=\mu^{\ast}(\mathscr{C})$. This finishes the proof.
\end{proof}

\noindent Therefore, $\bigotimes_{i\in\mathbb{N}}\mu_{i}$ is a measure on $\bigotimes_{i\in\mathbb{N}}\Sigma_{i}$ satisfying identity \eqref{E3}. In conclusion, we have proved the main result of this article, Theorem \ref{J}.
\vspace{10pt}
\par
\noindent 
\textbf{Acknowledgements.} The author expresses his deepest gratitude to the (anonymous) reviewer of this paper for his truly professional work.


\begin{thebibliography}{9}

\bibitem[1]{A}
\textsc{L. Arkeryd, N. Cutland, C.W. Henson}. 
Nonstandard analysis.
\emph{NATO Advanced Science Institutes Series C: Mathematical and Physical Sciences, vol. 493, Kluwer Academic Publishers Group, Dordrecht, (1997).}

\bibitem[2]{RB}
\textsc{R. Baker}. 
Lebesgue Measure on $\mathbb{R}^{\infty}$.
\emph{Procedings of the American Mathematical Society Volume 113 Number 4, pp. 1023-1029, (1991).}

\bibitem[3]{RB2}
\textsc{R. Baker}. 
Lebesgue Measure on $\mathbb{R}^{\infty}$, II.
\emph{Procedings of the American Mathematical Society Volume 132 Number 9, pp. 2577-2591, (2004).}

\bibitem[4]{B}
\textsc{V. I. Bogachev}. 
Measure Theory Vol I and II.
\emph{Springer-Verlag Berlin Heidelberg (2007).}

\bibitem[5]{EM}
\textsc{E. O. Elliot, A. P. Morse}. 
General Product Measures.
\emph{Transactions of the American Mathematical Society Volume 110, pp. 245–283, (1963).}

\bibitem[6]{F}
\textsc{G. B. Folland}. 
Real Analysis: Modern Thechniques and their Applications.
\emph{John Wiley \& Sons (1999).}

\bibitem[7]{H}
\textsc{E. Hopf}. 
Ergodentheorie.
\emph{Berlin, pp. 2, (1937).}

\bibitem[8]{K1}
\textsc{S. Kakutani}. 
Notes on Infinite Product Measure Spaces, I.
\emph{Proc. Imp. Acad. Volume 19 Number 3, pp. 148-151, (March 1943).}

\bibitem[9]{K2}
\textsc{A. Kolmogoroff}. 
Grundbegriffe der Wahrscheinlichkeitsrechnung.
\emph{Berlin (1933).}

\bibitem[10]{LR}
\textsc{P. Loeb, P. Ross}. 
Infinite Products of Infinite Measures.
\emph{Illinois Journal of Mathematics 1, pp. 153-158, (Spring, 2005).}

\bibitem[11]{LU}
\textsc{Z. Lomnicki, S. Ulam}.
Sur la theorie de la mesure dans les espaces combinatoires et son application au calcul des probabilites. I: Variables independantes.
\emph{Fund. Math. vol. 23, pp. 237-278, (1934).}

\bibitem[12]{P1}
\textsc{G. Pantsulaia}. 
On Ordinary and Standard Products of Infinite Family of $\sigma$-finite Measures and Some of Their Applications.
\emph{Acta Mathematica Sinica, English Series Volume 27 Number 3, pp. 477–496, (February 2011).}

\bibitem[13]{R}
\textsc{R. A. Rogers}.
Hausdorff Measures.
\emph{Cambridge University Press, (1970)}.

\bibitem[14]{S1}
\textsc{S. Saeki}. 
A Proof of the Existence of Infinite Product Probability Measures.
\emph{The American Mathematical Monthly Vol. 103, No. 8, pp. 682-683, (October, 1996).}

\bibitem[15]{T}
\textsc{A. N. Tychonoff}.
Über die topologische Erweiterung von Räumen.
\emph{Mathematische Annalen, 102 (1), pp. 544–561 (1930).}

\bibitem[16]{Y}
\textsc{Y. Yamasaki}. 
Measures on Infinite Dimensional Spaces.
\emph{World Scientific (1985).}
\end{thebibliography}
\end{document}